\documentclass[reqno,english]{amsart}
\usepackage{algorithm,enumerate}
\usepackage{algpseudocode}

\usepackage{amsmath}
\usepackage[hidelinks]{hyperref}

\usepackage{enumerate,amsmath,amsthm,latexsym,amssymb}
\usepackage{graphicx}

\def\cS{{\mathcal S}}
\def\cC{{\mathcal C}}

\def\cL{{\mathcal L}}
\def\cP{\mathcal{P}}

\def\f{\phi}

\def\cE{\mathcal{E}}
\def\cL{\mathcal{L}}

\def\cT{\mathcal{T}}

\renewcommand{\Pr}{\operatorname{\bf Pr}}
\newcommand\bfrac[2]{\left(\frac{#1}{#2}\right)}

\newtheorem{thm}{Theorem}

\newtheorem{theorem}{Theorem}[section]

\newtheorem{lemma}[theorem]{Lemma}

\newtheorem{observation}[theorem]{Observation}

\newtheorem{notation}[theorem]{Notation}

\begin{document}
\title{A note on long cycles in sparse random
graphs}
\author{Michael Anastos}
\address{M.\ Anastos\hfill\break
Institute of Science and Technology Austria \\Klosterneurburg 3400, Austria.}
\email{michael.anastos@ist.ac.at}

\thanks{ This project has received funding from the European Union's Horizon 2020 research and innovation
programme under the Marie Sk\l{}odowska-Curie grant agreement No 101034413
\includegraphics[width=5.5mm, height=4mm]{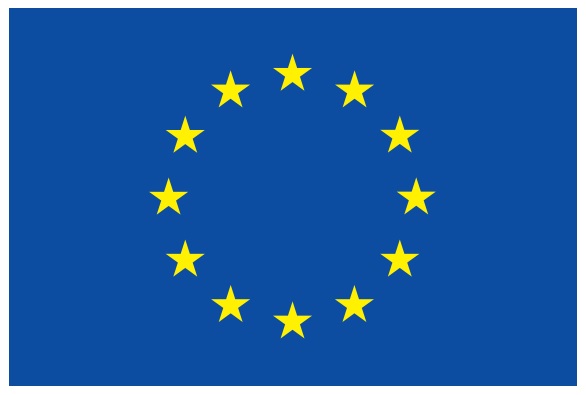}.
}

\maketitle
\begin{abstract}
Let $L_{c,n}$ denote the size of the longest cycle in $G(n,{c}/{n})$, $c>1$ constant. We show that there exists a continuous function $f(c)$ such that $ L_{c,n}/n \to f(c)$ a.s. for $c\geq 20$, thus extending a result of Frieze and the author to smaller values of $c$. Thereafter, for $c\geq 20$, we determine the limit of the probability that $G(n,c/n)$ contains cycles of every length between the length of its  shortest and its longest cycles as $n\to \infty$. 
\end{abstract}
\section{Introduction}
Let $L_{c,n}$ denote the size of the longest cycle in $G(n,p)$, $p={c}/{n}$ i.e. the random graph on $[n]$ where each edge appears independently with probability $p$. Erd\H{o}s \cite{erdos1975problems} conjectured that if $c>1$ then w.h.p.\footnote{We say that a sequence of events $\{\mathcal{E}_n\}_{n\geq 1}$
holds {\em{with high probability}} (w.h.p.\@ in short) if $\lim_{n \to \infty}\Pr(\mathcal{E}_n)=1-o(1)$.}  $L_{c,n}\geq \ell(c)n$ where $\ell(c)>0$ is independent of $n$. This was proved by Ajtai, Koml\'os and Szemer\'edi \cite{ajtai1981longest} and in a slightly weaker form by  Fernandez de la Vega \cite{de1979long} who proved that the conjecture is true for $c>4\log 2$ ($\log $ is in base $e$). Although this answers Erd\H{o}s's question and provides the order of magnitude of $L_{c,n}$ for $c>1$ it leaves open the question of providing matching upper and lower bounds on $L_{c,n}$ up to the linear in $n$ order term.  Bollob\'as \cite{bollobas1982long} realized that for large $c$ one could find a large path/cycle w.h.p. by concentrating on a large subgraph with large minimum degree and demonstrating Hamiltonicity.
In this way he showed that $L_{c,n}\geq (1-c^{24}e^{-c/2})n$ w.h.p. This was then improved by Bollob\'as, Fenner and Frieze \cite{bollobas1984long} to $L_{c,n}\geq (1-c^6e^{-c})n$  and then by Frieze \cite{frieze1986large} to $L_{c,n}\geq ( 1-(1+\epsilon_c)(1+c)e^{-c})n$ w.h.p. where $\epsilon_c\to0$ as $c\to\infty$. This last result is optimal up to the value of $\epsilon_c$, as there are  $(1+c)e^{-c}n+o(n)$ vertices of degree 0 or 1 w.h.p. Finally the scaling limit of $L_{c,n}$ was determined by Anastos and Frieze \cite{anastos2021scaling} for sufficiently large $c$.
They  showed that there exists some absolute constant $C_0>1$ and a function $f(\cdot)$ such that for $c\geq C_0$, $L_{c,n}/n\to f(c)$ a.s. In addition they gave a way of computing  $L_{c,n}$ within arbitrary accuracy.  
They also proved analogous results for the longest direct cycle in sparse random digraphs \cite{anastos2022scaling}.

Denote by $L_{c,n}^P$ the length of the longest path in $G(n,p)$. In addition for a graph $G$ denote by $L(G)$ the size of the longest cycle of $G$. The main theorem of this paper is the following one. 
\begin{theorem}\label{thm:main}
Let $G\sim G(n,c/n)$.   
\begin{itemize}
    \item[(a)] There exists a continuous function $f:[0,\infty) \to [0,1]$ such that $L_{c,n}/{n} \to f(c)$ almost surely for $c\geq 20$, constant.
    \item[(b)] W.h.p. $G$ has a cycle of length $L(G)-i$ for $0\leq i\leq 0.1{c^3e^{-c}n}$, for $20\leq c \leq 0.4\log n$.
    \item[(c)] W.h.p. $|L_{c,n}-L_{c,n}^P|\leq  ({2000\log n})/{c}+1$ for $20\leq c \leq 0.4\log n$.
\end{itemize}
\end{theorem}
We discuss the case $c\geq 0.4\log n$ shortly. Part (a) of Theorem \ref{thm:main}, except for the continuity of $f$, is proven in \cite{anastos2021scaling} for sufficiently large $c$. The proof there relies on identifying a subgraph $H_3$ of $G$ and showing that after contracting every maximal path whose interior vertices are of degree 2 into a single edge we get a graph $H_3'$ of minimum degree 3 with the property that it has a Hamilton cycle that passes through all the ``new'' edges. To find the Hamilton cycle there a version of the coloring argument of Fenner and Frieze \cite{fenner1984hamiltonian} is used. For the corresponding calculations the condition $c\geq 10^6$ is asserted. Our improvement on $c$ comes from considering a subgraph $H_4$ of $G$ in place of $H_3$. $H_4$ is constructed in a similar manner as $H_3$. The alterations done to its construction are such that $H_4'$, the graph obtained after contracting every maximal path  whose interior vertices are of degree 2  into an edge, has minimum degree 4. This enables us to use a different argument to find a suitable Hamilton cycle in $H_4'$ and thus extend the range of $c$ for which part (a) of Theorem \ref{thm:main} is true to $c\geq 20$. 

A graph $G$ is called pancyclic if it contains a cycle of length $\ell$ for every $\ell \in [3,|V(G)|]$. The study of pancyclic graphs was initiated by Bondy \cite{bondy1971pancyclic}.  Cooper and Frieze \cite{cooper1987pancyclic} proved that the threshold for $G(n,p)$ being pancyclic is the same as being hamiltonian which is $p_H=({\log n +\log\log n})/{n}$. Their methods can be extended to prove that the probability of the 2-core of $G(n,p)$ being pancyclic is the same as being hamiltonian which is $1-o(1)$ for $p \geq {(1+\epsilon)\log n}/{3n}$ for any constant $\epsilon>0$. Thus for $c\geq 0.4\log n$ the size of the longest cycle in $G(n,p)$ equals to the size of its $2$-core, say $n'$, and there exists a cycle of length $\ell$ in $G$ for every  $\ell\in [3,n']$ w.h.p. 

For a fixed set $S\subset \mathbb{N}\setminus \{1,2\}$ the probability that $G(n,c/n)$ contains a cycle of length $l$ for $l\in S$  is given by a result of Bollob\'as (see \cite{bollobas2001random}, §4.1) and separately by a result of Karo\'nski and Ruci\'nski \cite{karonski1983number}. For $l\geq 3$ let $Z_{c,l}$ be the number of cycles of length $l$ in $G(n,c/n)$. They  proved that for every finite set $S\subset \mathbb{N}\setminus \{1,2\}$ the joint distribution of $\{Z_{c,l}:l\in S\}$ converges in distribution to the joint distribution of $\{Poisson({c^l}/{2l}):l\in S\}$. Recently, Alon, Krivelevich and Lubetzky \cite{alon2021cycle} studied the set of cycle lengths of randomly augmented graphs and  showed that if one sprinkles $\epsilon n$ random edges on top of some graph $G$ on $[n]$ then, in addition to $L(G)$, the new graph w.h.p. contains a cycle of length $\ell$ for every $\ell$ such that both $\ell, L(G)-\ell$ tend to infinity with $n$. For graphs $G,F$ on the same vertex set denote by $G \oplus F$ the graph $(V(G),E(G)\cup E(F))$.
\begin{thm}[Theorem 3.1 of \cite{alon2021cycle}]\label{thm:AKL}
Fix $\delta>0$, let $H$ be a graph on $[n]$ with a longest cycle of size $L(H)$, $F \sim  G(n, {\delta}/{n})$ and $G=H \oplus F$. There exist absolute constants $C_1,C_2>0$ such that, for any $3\leq \ell \leq |V(H)|/2$, we have that 
 $G$ contains a cycle of length $l$ for every $l\in [\ell,L(H)-\ell+4]$ with probability at least $1-C_1e^{-C_2(\delta^2\wedge 1)\ell}$.
\end{thm}

To generalize the notion of pancyclic graphs Brandt \cite{brandt1997sufficient} introduced the notion of weakly pancyclic graphs. A graph $G$ is \emph{weakly pancyclic} if it contains cycles of every length between the lengths of its shortest and longest cycles. In the following theorem we study the distribution of the set of cycles lengths of $G(n,p)$ and determine the limit of the probability that it is weakly pancyclic as $n\to \infty$. 
\begin{theorem}\label{thm:weaklypanc}
Let $G\sim G(n,c/n)$, $c\geq 20$. Then for every $S\subseteq [L(G)] \setminus \{1,2\}$,
\begin{equation}\label{eq:probweakly}
    \lim_{n\to \infty} \Pr( G  \text{ contains a cycle of length $l$ for }l\in S) = \prod_{k\in S}\bigg(1-e^{-\frac{c^k}{2k}}\bigg).
\end{equation}
In particular, 
\begin{equation}\label{eq:probweakly2}
\lim_{n\to \infty} \Pr( G  \text{ is weakly pancyclic }) = \sum_{k\geq 3}\prod_{\ell=3}^{k-1}e^{-\frac{c^{\ell}}{2\ell}}\prod_{\ell=k}^{\infty}\bigg(1-e^{-\frac{c^\ell}{2\ell}}\bigg).
\end{equation}
\end{theorem}
Observe that \eqref{eq:probweakly} is given by Bollob\'as and by Karo\'nski and Ruci\'nski in the case $\max S=O(1)$. In the proof of Theorem \ref{thm:weaklypanc} we make use of a weak lower bound on $L_{c,n}$ given by the following Lemma. Its proof is located at the end of Section \ref{sec:longcycles}.

\begin{lemma}\label{lem:weakbound}
 W.h.p. $n-0.04c^3e^{-c}n \leq L_{c,n}$ for $20\leq c \leq 0.4\log n$.  
\end{lemma}
 
\textbf{Proof of Theorem \ref{thm:weaklypanc}:} For $c\geq 0.4\log n$ Theorem \ref{thm:weaklypanc} follows from the fact that the 2-core of $G(n,p)$ is pancyclic w.h.p. thus we may assume that $20\leq c<0.4\log n$.  Let $d$ be such that $n-0.05d^3e^{-d}n= n-0.1c^3e^{-c}n$. Then $d<c$, in particular $c-d=\Omega(1)$. We may generate $G$ by letting $G_1\sim G(n, {d}/{n})$, $G_2\sim G(n,p')$ and $G=G_1\oplus G_2$ where $(1-c/n)=(1-p')(1-d/n)$. Let $\epsilon>0$ and $\ell$ be the minimum positive integer such that $C_1e^{-C_2(\delta^2\wedge 1)\ell}<\epsilon$ and $\prod_{k=\ell+1}^{\infty} (1-e^{-\frac{c^k}{2k}})>1-\epsilon$ where the constants $C_1,C_2$ are as in the statement of Theorem \ref{thm:AKL}. Also denote by $\cL(G)$ the set $\{l\in [n]:G \text{ spans a cycle of length } l\}$.

Lemma \ref{lem:weakbound} applied to $G_1$ and Theorem \ref{thm:AKL} applied to $G_1\oplus G_2$ give that $G$ contains a cycle of length $l$ for every integer $l\in [\ell, n-0.05d^3e^{-d}n] = [\ell, n-0.1{c^3e^{-c}n}]$ with probability at least $1-C_1e^{-C_2(\delta^2\wedge 1)\ell}+o(1)$. On the other hand part (b) of Theorem \ref{thm:main} implies that $\cL(G)$ contains the integers in $[L(G)-0.1{c^3e^{-c}n}, L(G)]$ w.h.p. As $L(G)-0.1{c^3e^{-c}n} \leq n- 0.1{c^3e^{-c}n}$ we have that $\cL(G)$ contains $[\ell,L(G)]$ with probability at least $1-C_1e^{-C_2(\delta^2\wedge 1)\ell}+o(1)$. Combining this last statement with the results of Bollob\'as and of Karo\'nski and Ruci\'nski gives \eqref{eq:probweakly}. Indeed for $S\subset [L(G)]\setminus \{1,2\}$,
\begin{align*}
 \prod_{k\in S}\bigg(1-e^{-\frac{c^k}{2k}}\bigg)
 +\epsilon& \geq    \prod_{k\in S\cap [\ell]}\bigg(1-e^{-\frac{c^k}{2k}}\bigg)=
    \lim_{n\to \infty} 
    \Pr( S\cap [\ell] \subseteq \cL(G))
\\  & \geq \lim_{n\to \infty} 
    \Pr( S \subseteq \cL(G))
      \geq \lim_{n\to \infty} 
    \Pr( S\cap [\ell] \subseteq \cL(G) \text{ and } [\ell+1,L(G)]\subseteq L(G))
    \\&\geq \prod_{k\in S\cap [\ell]}\bigg(1-e^{-\frac{c^k}{2k}}\bigg)-
 C_1e^{-C_2(\delta^2\wedge 1)\ell}\geq \prod_{k\in S}\bigg(1-e^{-\frac{c^k}{2k}}\bigg)-\epsilon.
\end{align*}
Similarly one can derive \eqref{eq:probweakly2}; the summation at \eqref{eq:probweakly2} corresponds to the sum over $k$ of the probabilities that $G$ is weakly pancyclic and has girth $k$. 
\qed

The proof of Theorem \ref{thm:main} relies on the study of an induced subgraph of $G$ and how it lies in $G$ which we relate to a subset $S$ of $V(G)$ which we call the \emph{strong $4$-core} of $G$. We define the strong $4$-core of $G$ and establish some of its basic properties in Section \ref{sec:strong}. Using the strong $4$-core we identify an induced subgraph $F$ of $G(n,p)$ such that  no subgraph of $G$ that spans more vertices can be hamiltonian. We then prove that $F$ is hamiltonian and derive parts (b) and (c) of Theorem \ref{thm:main}. This, modulo the  Hamiltonicity argument which is presented at Section \ref{sec:Ham}, is presented at Section \ref{sec:longcycles}. Finally, for the sake of completeness, at Section \ref{sec:scalinglimit} we present the proof of part (a) of Theorem \ref{thm:main}.

\section{Preliminaries and Notation}
For a graph $G$ we denote by $V(G)$ and $E(G)$ its vertex set and edge set respectively. For $v\in V(G)$ and $k\in \mathbb{N}$  we denote by $N^k(v), N^{<k}(v)$ and $N^{\leq k}(v)$ the set of vertices  within distance exactly $k$, less than $k$ and at most $k$ respectively from $v$ in $G$. For $U\subseteq V(G)$ we let $N(U)$ be the set of vertices in $V(G)\setminus U$ that are adjacent to $U$ and $G[U]$ be the subgraph of $G$ induced by $U$. For $M\subseteq \binom{V(G)}{2}$ we let $G\cup M= (V(G),E(G)\cup M)$ and $G\setminus M= (V(G),E(G)\setminus M)$. We denote by $\delta(G)$ and $\Delta(G)$ the minimum and maximum respectively degree of $G$. Finally by $\log x$ we denote the natural logarithm of $x$.

Throughout the paper we make use of Lemma \ref{lem:AzumaHoeffding}, an extension of McDiarmid’s inequality given by Warnke in \cite{warnke2016method} (see Theorem 1.2 and Remark 2). Compared to the more general Theorem 1.2 of \cite{warnke2016method}, Lemma \ref{lem:AzumaHoeffding} is restated in a form that is easier to apply in our setting. For the reduction of Lemma \ref{lem:AzumaHoeffding} from Theorem 1.2 of \cite{warnke2016method} we let $G\sim G(n,p)$ with $np\leq 2\log n$, consider the vertex exposure martingale for revealing $G$ and make use of the fact that $G$ has maximum degree smaller than $\log^2 n$ with probability $1-o(n^{-10})$.

\begin{lemma}\label{lem:AzumaHoeffding}
Let $G\sim G(n,p)$ with $np\leq 2\log n$. Let $f$ be a graph theoretic function such that $|f(G')|\leq n$ for every graph $G'$ of order $n$. Assume that there exists an integer $d=d(n)$ with the property 
that for every $v\in [n]$ and every graph $G_1$ on $[n]$ of maximum degree $\log^2 n$, with $G_2$ being the graph obtained from $G_1$ by deleting  all the edges incident to $v$, we have that 
$$ |f(G_1)-f(G_2)|\leq 0.5d.$$
Then for every $t>0$,
\begin{align}\label{eq:Azuma}
\Pr ( | f(G)-\mathbb{E}(f(G)) |> t ) \leq 2\exp \left( -\frac{t^2}{2n(d+1)^2} \right) + o(n^{-8}).
\end{align}
\end{lemma}

\section{The strong \texorpdfstring{$k$-core}{Lg}}\label{sec:strong}

The $k$-core of $G$ is the induced subgraph of $G$ whose vertex set is the maximal subset $S$ of $V(G)$ with the property that every vertex in $S$ has at least $k$ neighbors in $S$. It is well known to be unique and it can be obtained by iteratively removing from $G$ vertices with fewer than $k$ neighbors among the vertices left. The concept of the $k$-core was introduced by Bollob\'as in his study of the evolution of sparse graphs \cite{bollobas1984evolution}. Some years later, Pittel, Spencer and Wormald \cite{pittel1996sudden} proved that the property of having a nonempty $k$-core has a sharp threshold in the random graph model $G(n,p)$. Namely the proved that there exists a constant $c_k$ such that $G(n,c/{n})$ has a nonempty $k$-core with probability $1-o(1)$ if $c>c_k$ and with probability $o(1)$ if $c<c_k$. In addition they gave a way of calculating $c_k$.

To identify the vertex set of a longest cycle in $G(n,p)$ we use a concept similar to that of the $k$-core. For a graph $G$ we define the \emph{strong $k$-core} of $G$ to be the maximal subset $S$ of $V(G)$ with the property that every vertex in $S\cup N(S)$ has at least $k$ neighbors in $S$. Observe that if the sets $S_1,S_2\subset V(G)$ have this property then so does the set $S_1\cup S_2$. Thus the strong $k$-core of a graph is well-defined. It can also be obtained via the following red/blue/black coloring procedure:

\begin{algorithm}[H]
\caption{}
\begin{algorithmic}[1]
\\ Input: a graph $G$, an integer $k$.
\\ Initially color all the vertices of $G$ black.  \While{  there exists a black or blue vertex $v\in V(G)$ with fewer than $k$ black neighbors}
\\\hspace{5mm} Color $v$ red and its black neighbors blue.
\EndWhile
\\ Return the coloring of $G$.
\end{algorithmic}
\end{algorithm}

For a graph $G$ we let $V_{k,black}(G), V_{k,blue}(G)$ and $V_{k,red}(G)$ be the set of vertices whose final color given by Algorithm 1 is black, blue and red respectively. Also denote by $SC_k(G)$ the vertex set of its strong $k$-core. Observe that the set $V_{k,black}(G)$ has the property that no vertex in $V_{k,black}(G)\cup N(V_{k,black}(G))$ is red. Therefore every vertex $v\in V_{k,black}(G)\cup N(V_{k,black}(G))$ has at least $k$ neighbors in $V_{k,black}(G)$. Consequentially, $V_{k,black}(G) \subseteq SC_k(G)$. On the other hand no vertex in $SC_k(G)$ would  ever be colored red or blue. Indeed assume otherwise and let $v$ be the first vertex in $SC_k(G)$ that receives a color red or blue. If that color is red then at that moment $v$ has fewer than $k$ black neighbors. Else if $v$ receives color blue then it has a neighbor $u$ that receives color red and therefore at that moment $u$ has less than $k$ black neighbors. As in both cases $SC_k(G)$ is a subset of the set of black vertices at the moment that $v$ receives a color other than black we get a contradiction. 

For the rest of this paper we will denote by $V_{black}(G)$ the vertex set of the strong $4$-core of $G$, by $V_{blue}(G)$  the neighborhood of $V_{black}(G)$ and by $V_{red}(G)$ the rest of the vertices of $G$. We call the vertices in $V_{black}(G)$, $V_{blue}(G)$ and $V_{red}(G)$, black, blue and red respectively.  In addition we  denote by $G^{r/b}$ the subgraph of $G$ induced by $V_{blue}(G) \cup V_{red}(G)$. A crucial observation about the structure of the subgraph of $G$ induced by $V_{blue}\cup V_{red}(G)$ is the following one.
\begin{observation}\label{obs:redvsblue}
During the execution of Algorithm 1 with inputs $G,4$, every time a vertex is colored red at most $3$ of its neighbors are colored blue. Thus every component $C$ of $G^{r/b}$ contains at least $\frac{|C|}{4}$ red vertices. These vertices do not have any neighbor outside $C$. 
\end{observation}
In the following Lemma we summarize the  properties of the strong 4-core of a random graph that we are going to use later on.
\begin{lemma}\label{lem:expsizecomp}
Let $G\sim G(n,{c}/{n})$, $20\leq c \leq 2\log n$. For $i\geq 1$ let $X_i$ be the number of vertices in $G$ that lie in components of size $i$ in $G^{r/b}$. Then the following hold with probability $1-o(n^{-2})$.
\begin{itemize}
    \item[(a)] $\mathbb{E}(X_i)\leq 0.8^{-i}n/(ci)$ and $X_i\leq 0.8^{-i}n/(ci)+n^{0.55}$ for $1\leq i\leq \log^3n$.
    \item[(b)] $X_i=0$ for $i\geq (10^3\log  n)/{c}$.
    \item[(c)] At most $0.03c^3e^{-c}n$ red vertices lie in a component of $G^{r/b}$ with at least 2 red vertices.
    \item[(d)] $|V_{red}(G)| \leq 0.25c^3e^{-c}n$ and $|V_{red}(G)\cup V_{blue}(G)| \leq  c^3e^{-c}n$.
\end{itemize} 
\end{lemma}
\begin{proof}
(a) Observation \ref{obs:redvsblue} implies that for every component of size $i$ we can identify sets $S,T$ with $|S|\geq i/4$, $|S|+|T|=i$ such that $G$ spans a tree on $S\cup T$ and no vertex in $S$ has a neighbor outside $S\cup T$. Therefore for $i\geq 1$,
\begin{align*}
    \mathbb{E}(X_i)&\leq i \binom{n}{i}i^{i-2}p^{i-1}\binom{i}{i/4}(1-p)^{\frac{i(n-i)}{4}} \leq  \bfrac{en}{i}^ii^{i-1}p^{i-1} 2^i    e^{-\frac{pi(n-i)}{4}}
\\&    \leq \frac{n}{ci} \bigg(2enpe^{-(0.25+o(1))c} \bigg)^i  \leq \frac{0.8^{-i}n}{ci}.
\end{align*}
At the last inequality we used that $c\geq 20$. For $v\in [n]$ deleting all the edges incident to $v$ in $G$ may increase or decrease the number of components of $G^{r/b}$ of size $i$ by at most $d(v)+1\leq \Delta(G)+1$ (any ``new" component contains an endpoint of a deleted edge). Therefore, Lemma \ref{lem:AzumaHoeffding} implies that $X_i\leq 0.8^{-i}n/(ci)+n^{0.55}$ for $1\leq i\leq \log^3n$ with probability $1-o(n^{-2})$.

(b) From the above calculation we also get,
\begin{align*}
    \Pr\bigg(\sum_{i=\frac{10^3\log n}{c}}^{\log^3 n}X_i>0\bigg)
\leq \mathbb{E}\bigg(\sum_{i=\frac{10^3\log n}{c}}^{\log^3 n}X_i\bigg) \leq \sum_{i=\frac{10^3\log n}{c}}^{\log^3 n} n \bigg(2ece^{-0.235c} \bigg)^i e^{-0.01ci} =O(n^{-9})
\end{align*}
Now assume that $G^{r/b}$ has a component $C$ of size larger than $\log^3 n$. For $t\geq 0$ let $m_t$ be the largest component spanned by the vertices of $C$ that are either red or blue right after the $t^{th}$ time the while-loop of Algorithm 1 is executed.  Since at every step of our process a single vertex is colored red we have that $m_{t+1}\leq 1+ \Delta(G)\cdot \max\{m_t,1\}$. Thus either $\Delta(G)\geq \log^{1.5}n-1$ or there exists $t\geq 0$ such that $\log^{1.5} n\leq m_t \leq \log^3 n$. In the second case at time $t$ the vertices of $C$ span a component $C'$ on $m_t$ vertices with at least $m_t/4$ red vertices. Those red vertices have no neighbor outside $C'$ in $G$. Therefore $G^{r/b}$ spans a component of size at least $\log^3 n$ with probability at most
$$ O(n^{-9})+ \sum_{i=\log^{1.5} n}^{\log^3 n} n   \bigg(2enpe^{-(0.25+o(1))c} \bigg)^i +\Pr(Bin(n,p)\geq \log^{-1.5}n-1) =o(n^{-2}). $$

(c),(d) Let $Y$ and $Y_i$, $i\geq 1$ be the number of red vertices that lie in a component of $G^{r/b}$ with at least $2$ and exactly $i$ respectively red vertices. Then, $Y= Y_2+\sum_{i\geq 3} Y_i$. A component of $G^{r/b}$ with exactly 2 red vertices consists either of two adjacent vertices $u,v$ that have at most 5 neighbors in total in $[n]\setminus \{u,v\}$ or two non-adjacent vertices $u,v$ that have a common neighbor $w$ and at most 6 additional neighbors in total in $[n]\setminus \{u,v,w\}$. Therefore,
\begin{align*}
    \mathbb{E}(Y_2)&\leq 2\binom{n}{2}p \sum_{i=0}^5 \binom{n}{i} 2^ip^i (1-p)^{2(n-2-i)}+  2\binom{n}{2} n p^2 \sum_{i=0}^6 \binom{n}{i} 2^ip^i (1-p)^{2(n-3-i)}
    \\&\leq (1+o(1))cn \sum_{i=0}^5 \frac{2^ip^{i}n^i}{i!} e^{-2pn}+ (1+o(1)) c^2n  \sum_{i=0}^6 \frac{2^ip^in^i}{i!}  e^{-2pn}
    \\&\leq  (1+o(1))c^2e^{-2c}n\bigg( \sum_{i=0}^5 \frac{(2c)^i}{i!c}+ \sum_{i=0}^6 \frac{(2c)^i}{i!} \bigg)
    \leq c^2e^{-2c}n\cdot \frac{2c^6}{6!} \leq 10^{-4}c^3e^{-c}n. 
\end{align*}
Thereafter, similarly to the calculation of $\mathbb{E}(X_s)$ we have,
\begin{align*}
    \sum_{s=3}^{\log^3 n} \mathbb{E}( Y_s) &\leq \sum_{s = 3}^{\log^3 n}\sum_{t=0}^{3s} s\binom{n}{s+t} \binom{s+t}{t} (s+t)^{s+t-2}p^{s+t-1} (1-p)^{(n-s-t)s }
    \\&\leq \sum_{s= 3}^{\log^3 n} \bigg(\sum_{t=0}^{3s} c^{-t}\bigg) s\binom{n}{4s} \binom{4s}{3s} (4s)^{4s-2}p^{4s-1} e^{-p(n-4s)s }
    \\&\leq \sum_{s= 3}^{\log^3 n}  \frac{1.1 (4cs)^{4s-1}  e^{-cs } n} {(3s)!s!} 
    \leq \sum_{s = 3}^{6} \frac{1.1 (4cs)^{4s-1}  e^{-cs } n} {(3s)!s!}  + \sum_{s= 7}^{\log^3 n} \frac{1.1n}{4cs} \bfrac{e^{4}c^{4} 4^{4} e^{-c}} {3^3}^s
    \\& \leq 0.015c^3e^{-c}n +c^3e^{-c}n\sum_{s= 7}^{\log^3 n}  \frac{1.1e^44^4}{3^3 4s} \bfrac{e^{4}c^{4} 4^{4} e^{-c}} {3^3}^{s-1}
   \\& \leq 0.015c^3e^{-c}n +c^3e^{-c} n \sum_{s= 7}^{\log^3 n}  \frac{142.5}{s} \cdot 0.175^s \leq 0.02c^3e^{-c}n.
\end{align*}
Lastly, $Y_1$ is bounded above by the number of vertices of degree $0,1,2$ or $3$ in $G$. Therefore, $\mathbb{E}(Y_1)\leq (1+c+0.5c^2+c^3/6)e^{-c}n\leq 0.2c^3e^{-c}n$.

For $v\in [n]$ deleting all the edges incident to $v$ in $G$ may increase or decrease the number of components of $G^{r/b}$ with exactly $i$ red vertices by at most  $d(v)+1\leq \Delta(G)+1$ (any ``new" component contains an endpoint of a deleted edge). Therefore, part (b) of this lemma and Lemma \ref{lem:AzumaHoeffding} imply that $|V_{red}(G)|=\sum_{i=1}^{\log^3 n} Y_i \leq 0.25c^3e^{-c}n$ and 
$Y=\sum_{i=2}^{\log^3 n} Y_i \leq 0.03c^3e^{-c}n$ with probability $1-o(n^{-2})$. Finally, by Observation \ref{obs:redvsblue}, $|V_{red}(G)\cup V_{blue}(G)|\leq 4|V_{red}(G)|\leq c^3e^{-c}n$ with probability $1-o(n^{-2})$.
\end{proof}

For proving Theorem \ref{thm:main} we will use  Theorem \ref{thm:ham}. We apply Theorem \ref{thm:ham} in the next section while we present its proof in Section \ref{sec:Ham}. 

\begin{theorem}\label{thm:ham}
Let $G\sim G(n,c/n)$, $20\leq c$. Let $G'$ be the subgraph of $G$ induced by $V_{black}(G)\cup V_{blue}(G)$. Then for every $U \subseteq V_{blue}(G)$ and matching $M$ on  $V_{blue}\setminus U$ we have that $G'[V(G')\setminus U] \cup M$ contains a Hamilton cycle that spans all the edges in $M$  with probability $1-O(n^{-2})$.
\end{theorem}

Comparing the $k$-core with the strong $k$-core we note the following. The $k$-core can be thought as a procedure that separates the sparse from the denser portion of a graph. On the other hand the strong $k$-core is a procedure that separates the sparse (red) portion of the graph from a dense one (black) by a vertex cut (blue) while ensuring that the vertices in the cut are robustly connected to the dense part. This last key property is what enable us to extend a given matching $M$ on the blue vertices to a cycle that covers all of $M$, $V_{blue}(G)\setminus U$ and $V_{black}(G)$ in the proof of Theorem \ref{thm:ham}.

\section{Identifying the vertex set of a longest cycle}\label{sec:longcycles}
We start this section by showing how any red/blue/black vertex coloring of $G$ with the property that there does not exist a red to black edge can be used to upper bound  $L(G)$. We then use the red/blue/black coloring associated with the strong $4$-core (described in the previous section) to obtain an upper bound on $L(G(n,c/n))$ which will turn out to be tight. 

\begin{notation}\label{not:pathcover}
For a graph $G$ and a coloring $\gamma:V(G)\to \{red,blue,black\}$ we let $\mathcal{T}(G,\gamma)$ be the set of the components of the subgraph of $G$ induced by the $\gamma$-blue and $\gamma$-red  vertices. Thereafter, for $T\in \mathcal{T}(G,\gamma)$ we denote by $\cP_{T,\gamma}$ the set of all sets of vertex disjoint paths with $\gamma$-blue endpoints spanned by $T$. Here we allow paths of length 0. So a single blue vertex counts as a path. For $P\in \cP_{T,\gamma}$ let $n(T,\gamma,P)$ be the number of red vertices in $V(T)$ that are not covered by some path in $P$. Finally we let  $\f(T,\gamma)=\min_{P\in \cP_{T,\gamma}} n(T,\gamma,P)$.  
\end{notation}

\begin{lemma}\label{lem:upperbound}
For any red/blue/black coloring  $\gamma$ of $G$ with the property that there is no edge from a red to a black vertex we have,
\begin{equation}\label{upperbound}
    L(G)\leq |V(G)|-\sum_{T\in \cT(G,\gamma)}\f(T,\gamma).
\end{equation}
\end{lemma}
\begin{proof}
For any $T\in \mathcal{T}(G,\gamma)$ and any cycle $C$ of $G$ we have that $C$ induces a set of vertex disjoint paths on $V(T)$ with $\gamma$-blue endpoints. These paths leave uncovered at least $\phi(T,\gamma)$ many  $\gamma$-red vertices of $V(T)$. Hence any cycle of $G$ spans at most $n-\sum_{T\in \cT(G,\gamma)}\f(T,\gamma)$ vertices.  
\end{proof}

Henceforward we let $\gamma^*:V(G)\to \{\text{red, blue, black}\}$ be the coloring that colors the vertices of $V_{x}(G)$ with color $x$ for $x\in \{\text{red,blue,black}\}$. Recall that we refer to $\gamma^*$-red/blue/black vertices as simply red/blue/black vertices. For $T\in \mathcal{T}(G,\gamma^*)$ we fix a set of vertex disjoint paths with blue endpoints $P^*(T)$ with the property that $\cup_{T\in T(G,\gamma^*)} P^*(T)$ covers all but $\sum_{T\in \cT(G,\gamma^*)}\f(T,\gamma^*)$ red vertices. 
We also let $\cT(G)$ be the set of paths in $\cup_{T\in T(G,\gamma^*)} P^*(T)$ that cover a single red vertex. 

\begin{theorem}\label{basic}
Let $G\sim G(n,c/n), c\geq 20$. With probability $1-O(n^{-2})$,
\begin{equation}
 L(G) = n -\sum_{T\in \cT(G,\gamma^*)}\f(T,\gamma^*).
\end{equation}
In addition $G$ spans a cycle of length $L(G)-\ell$  with probability $1-O(n^{-2})$ for $0\leq \ell \leq |\cT(G)|$.
\end{theorem}
\begin{proof}
The inequality $L(G) \leq n -\sum_{T\in \cT(G,\gamma^*)}\f(T,\gamma^*)$ is given by Lemma \ref{lem:upperbound}. Indeed, as during Algorithm  1 every time a vertex is colored red its black neighbors are color blue we have that there is no edge from $V_{red}(G)$ to $V_{black}(G)$ and hence Lemma \ref{lem:upperbound} applies.

Now fix $\ell \in\{0,1,2,...,|\cT(G)|\}$ and let  $\{P_1,P_2,...,P_{|\cT(G)|} \}$ be an ordering of the paths in $\cT(G)$ (recall $\cT(G)$ is the set of paths in $\cup_{T\in T(G,\gamma^*)} P^*(T)$ that cover a single red vertex). Let $M(\ell)$ be the matching on $V_{blue}(G)$ obtained by replacing each path in $\cP(\ell):= \big(\cup_{T\in T(G,\gamma^*)} P^*(T)\big)\setminus \{P_i:i\in [\ell]\}$ by a single edge joining its endpoints. Also let $V_\ell^-$ be the set of vertices that lie in the interior of some path in $\cP(\ell)$. We define the graph $\Gamma(\ell)$ as follows. $V(\Gamma(\ell))=V_{black}(G) \cup(V_{blue}(G)\setminus V_\ell^-)$ and $E(\Gamma(\ell))$ consists of all the edges of $G$  spanned by  $V(\Gamma(\ell))$ plus the edges in $M(\ell)$.

Let $\mathcal{E}_\ell$ be the event that $\Gamma(\ell)$ contains a Hamilton cycle $C_\ell$ that spans all of the edges of $M(\ell)$. Assume that $\mathcal{E}_\ell$ occurs. Replace each edge of $C_\ell$ that belongs to $M(\ell)$ with the corresponding path in $\cP(\ell)$ and let $C_\ell'$ be the resulting cycle in $G$. Then $C_\ell'$ covers $V_{black}(G)\subseteq V(\Gamma_\ell)$. In addition, as $V_{blue}(G)\setminus V_\ell^- \subseteq V(\Gamma_\ell)$ and every vertex in $V_\ell^-$ lies in the interior of some path in $\cP(\ell)$, $C_\ell'$ covers $V_{blue}(G)$. Thereafter $C_\ell'$ also covers every vertex in  $V_{red}(G)$ that is covered by some path in $\cP(\ell)$. As  $\cP(\ell)= \big(\cup_{T\in T(G,\gamma^*)} P^*(T)\big)\setminus \{P_i:i\in [\ell]\}$, the set of vertex disjoint paths $\cup_{T\in T(G,\gamma^*)} P^*(T)$ covers $|V_{red}(G)|-\sum_{T\in \cT(G,\gamma^*)}\f(T,\gamma^*)$ vertices in $V_{red}(G)$ and each of the $\ell$ paths in $\{P_i:i\in [\ell]\}\subseteq \cup_{T\in T(G,\gamma^*)} P^*(T)$ covers a single vertex in $V_{red(G)}$ we have that 
$C_\ell'$ covers $|V_{red}(G)|-\sum_{T\in \cT(G,\gamma^*)}\f(T,\gamma^*)-\ell$ vertices in $V_{red}(G)$. All together $C_\ell'$ covers  $|V_{black}(G)|+|V_{blue}(G)|+|V_{red}(G)|-\sum_{T\in \cT(G,\gamma^*)}\f(T,\gamma^*)-\ell=n-\sum_{T\in \cT(G,\gamma^*)}\f(T,\gamma^*)-\ell$ vertices.

Theorem \ref{thm:ham} implies that $\Gamma(\ell)$ contains a Hamilton cycle that spans all of the edges of $M(\ell)$, hence $G$ spans a cycle of length $n-\sum_{T\in \cT(G,\gamma^*)}\f(T,\gamma^*)-\ell$, with probability $1-O(n^{-2})$. 
\end{proof}

\textbf{Proof of Lemma \ref{lem:weakbound}:} 
We construct a set of vertex disjoint paths in $G^{r/b}$ by taking 2 edges incident to every red vertex of degree at least 2 that lies in a component of $G^{r/b}$ containing a single red vertex. This set of edges induces a set of paths of length 2 with blue endpoints that do not cover red vertices in components with at least 2 red vertices and vertices of degree 0 or 1. $G$ has $(c+1)e^{-c}n +O(n^{-0.55})$ vertices of degree 0 or 1 w.h.p. (see \cite{frieze2016introduction}). Thus by Lemma \ref{lem:expsizecomp} they do not cover at most $(0.03c^3+c+1)e^{-c}+O(n^{0.6}) \geq 0.04c^3e^{-c}n $ red vertices w.h.p. Finally, Theorem \ref{basic} implies that $L_{c,n}\geq n-0.04c^3e^{-c}n$ w.h.p. 
\qed

\textbf{Proof of part (b) of Theorem \ref{thm:main}}
Given Theorem \ref{basic} it suffices to show that $|\cT(G)| \geq 0.1{c^3e^{-c}}n$ w.h.p. Every vertex of degree 3 lies in $V_{red}(G)$. Thus $|\cT(G)|$ is larger than the number of vertices of degree 3 minus the number of vertices that lie in a component of $G^{r/b}$ with at least 2 red vertices. $G$ has $c^3e^{-c}n/6 +O(n^{-0.55})$ vertices of degree 3 w.h.p. (see \cite{frieze2016introduction}).
Thus Lemma \ref{lem:expsizecomp} implies
$|\cT(G)| \geq (1/6-0.03)c^3e^{-c} +O(n^{-0.55}) \geq 0.1{c^3e^{-c}n}$ w.h.p.
\qed

\textbf{Proof of part (c) of Theorem \ref{thm:main}:} Similarly to the derivation of \eqref{upperbound}, any path $P$ of $G$ may cover at most $|V(G)|-\sum_{T\in \cT(G,\gamma^*)}\f(T,\gamma)+2r(G)$ vertices where $r(G)$ is the size of the largest component of $G^{r/b}$; $2r(G)$ is an upper bound on the number of vertices found in the first and last component of $G^{r/b}$ that intersects $P$ and we meet as we traverse $P$ from one of its endpoints to the other. Thus by Lemma \ref{lem:expsizecomp}, $L_{c,n}^P-L_{c,n}\leq 2\cdot \frac{1000\log n}{c}$. On the other hand 
$L_{c,n}^P-L_{c,n} \geq -1$ and therefore $|L_{c,n}^P-L_{c,n}|\leq \frac{2000\log n}{c}+1$.
\qed

\section{The scaling limit of the size of the longest cycle}\label{sec:scalinglimit}
To prove that ${L_{c,n}}/{n}$ has a limit $f(c)$ a.s. we first define a sequence of random variables $L_{c,n,k}$ that can be used to approximate $L_{c,n}$. 
Thereafter we show that for fixed $k\geq 1$ the sequence of random variables  $\{L_{c,n,k}/{n}\}_{n\geq 1}$ has a limit $f_k(c)$ a.s. This will imply that the sequence $\{f_k(c)\}_{k\geq 1}$ can be used to approximate $\lim_{n\to \infty} L_{c,n}/n$. In particular, it will imply that the sequence $\{f_k(c)\}_{k\geq 1}$ is a Cauchy sequence and therefore it has a limit $f(c)$. $f(c)$ will turn out to be the a.s. limit of $L_{c,n}/n$.
For the rest of this section we let $G\sim G(n,c/n)$, $c\geq 20$ constant.

\subsection{Approximating the longest cycle}
Before defining the random variables $L_{c,n,k}$, $k\geq 1$ we express the length of the longest cycle of $G$ as the sum of ``local" quantities. For this we introduce the following notation.

\begin{notation} 
For $v\in V(G)$ we let $\f(v)=0$ if $\gamma^*(v)= black$. Otherwise we let $\f(v)= {\f(T,\gamma^*|_T)}/{|T|} \in [0,1]$ where $T$ is the component of $G^{r/b}$ that contains $v$ and $\gamma^*|_T$ is the restriction of $\gamma^*$ on $T$.  
\end{notation}
Theorem \ref{basic} implies that with probability $1-o(n^{-2})$,
\begin{equation}\label{eq:local}
 L(G) = n -\sum_{T\in \cT(G,\gamma^*)}\f(T,\gamma^*)
 = n -\sum_{T\in \cT(G,\gamma^*)}|T|\cdot \frac {\f(T,\gamma^*)}{|T|}  
 = n -\sum_{v\in V(G)}\f(v).
\end{equation}
We now introduce the sequences of colorings $\{\gamma_k^*(v)\}_{k\geq 1}, v\in V(G)$ based on which we will define the local functions $\f_k:V(G)\to [0,1]$. We later use $\f_k$ to define $L_{c,n,k}$. For $v\in V(G)$ and $k\geq 1$  the coloring $\gamma_k^*(v):N_G^{\leq k}(v) \to\{red,blue,black\}$ is generated as follows. Initially all vertices in $N_G^{\leq k}(v)$ have color black. While there exists a blue or black vertex $u$ in $N_{G}^{<k}(v)$ with fewer than 4 black neighbors then color $u$ red and its black neighbors in  $N_{G}^{< k}(v)$ blue. 

For $k\geq 1$, given the colorings $\gamma_k^*(v),v\in [n]$ we define the function $\phi_k':V(G)\to [0,1]$ as follows.
 $\phi_k'(v)=0$ if $\gamma_k^*(v)= black$. Otherwise we let $\f'(v)={\f(T,\gamma^*_k)}/{|T|}$ where $T$ is the component containing $v$ in the subgraph of $G[N^{\leq k}(v)]$ induced by the $\gamma_k^*(v)$-red and $\gamma_k^*(v)$-blue vertices. Thereafter, given the function $\phi_k'$ we define the function $\f_k:V(G)\to [0,1]$ by $\f_k(v)=0$
if there exists $i\in [k]$ such that $|N^i(v)|\geq 10(ck)^{3i}$ or $G[N^{\leq k}(v)]$ spans a cycle and $\phi_k(v)=\phi_k'(v)$ otherwise. Finally we let
\begin{equation}\label{eq:localversion}
    L_{c,n,k}(G)=n-\sum_{v \in V(G)} \phi_k(v).
\end{equation}  
Equation \eqref{eq:local} implies,
\begin{align}\label{eq:localdiff}
|L(G)-L_{c,n,k}(G)|& \leq \sum_{v\in V(G)} \mathbb{I}(\phi_k(v)\neq \phi(v)) \leq \sum_{v\in V(G)} \mathbb{I}(\phi_k(v)\neq \phi_k'(v) \text { or } \phi_k'(v)\neq \phi(v)).
\end{align}

\begin{lemma}\label{lem:color1}
With probability $1-o(n^{-2})$, 
$$|\{v\in V(G):\phi_k(v)\neq \phi_k'(v) \text{ or } \phi_k'(v)\neq \phi(v)\}|\leq \frac{n}{4k^2}.$$
\end{lemma}
\begin{proof}
Let $X_k'$ be the set of vertices that lie in a component of $G^{r/b}$ of size at least $k$, $Y_k$  be the set of vertices that are within distance $k$ from a cycle of length at most $2k$ and $Z_k$ be the set of vertices with $|N^i(v)|\geq 10(ck)^{3i}$ for some $i\leq k$. We begin by showing that 
\begin{align}\label{eq:boundphi}
   |\{v\in V(G):\phi_k(v)\neq \phi_k'(v) \text{ or } \phi_k'(v)\neq \phi(v)\}|\leq |X_k'\cup Y_k\cup Z_k|\leq |X_k'|+|Y_k|+|Z_k|.
\end{align}
For that it is sufficient to show that (i) every vertex $v$ that is assigned the color black by $\gamma^*$ is also assigned the color black by $\gamma^*_k(v)$, thus $\phi(v)=\phi_k(v)=0$ and (ii) 
every vertex in $[n]\setminus (Y_k \cup Z_k)$ that lies in a component of size at most $k-1$ in $G^{r/b}$ satisfies $\phi_k'(v)=\phi(v)$. For $(i)$ note that the set of vertices that is assigned color black by  $\gamma^*_k(v)$ in $N^{<k}(v)$ is the maximal subset $S$ of $N^{<k}(v)$ such that every vertex in $S\cup N(S)$ has at least 4 neighbors in $S\cup N^k(v)$. On the other hand if we let $N^k_{black}(v)=N^k(v) \cap V_{black}(G)$ we have that  the set of vertices that is assigned color black by  $\gamma^*$ in $N^{<k}(v)$ is the maximal subset $S'$ of $N^{< k}(v)$ such that every vertex in $S'\cup N(S')$ has at least 4 neighbors in $S'\cup N^k_{black}(v)$. As $N^k_{black}(v)\subseteq N^k(v)$ we have that $S' \subseteq S$ and (i) follows. 

Now let $v \in [n]$ be a vertex that lies in a component $C$ of size at most $k-1$ in $G^{r/b}$ and $N(C)$ be the neighborhood of the vertices in $C$ in the graph $G$. Then every vertex in $N(C)$ is assigned color black by $\gamma^*$ and thus by $\gamma^*_k(v)$, by (i). Thus the set of black vertices in both colorings $\gamma^*,\gamma^*_k(v)$ that lie in $C$ is the maximal subset $S$ of $V(C)$ such that every vertex in $S\cup N(S)$ has at least 4 neighbors in $S\cup N(C) $. Therefore in both colorings every vertex in $N(C)$ receives color black and no vertex in $C$ receives color black. Thereafter the set of blue vertices with respect to either $\gamma^*$ or $\gamma_k^*(v)$ equals to the set of vertices with at least $4$ neighbors in $N(C)$, call this set $A$. Finally the vertices in $C\setminus A$ receive color red from both $\gamma^*$, $\gamma_k^*(v)$. Hence both $\gamma^*,\gamma^*_k(v)$ restricted to $C\cup N(C)$ are identical and therefore $\phi(v)=\phi_k(v)$.

We now bound $|X_k'|,|Y_k|$ and $|Z_k|$. Lemma \ref{lem:expsizecomp} implies that $|X_k'|\leq \sum_{i\geq k}\frac{n}{20i}0.8^i +O(n^{0.6}) \leq \frac{n}{10k^2}$ with probability $1-o(n^{-2})$. Thereafter let $Y_k'$ be the set of vertices that lie on a cycle of size at most $2k$. Then every vertex in $Y_k$ lies within distance at most $2k$ from a vertex in $Y_k'$ and therefore $|Y_k|\leq |Y_k'| \Delta^{2k}(G)$.
\begin{align*}
    \mathbb{E}(|Y_k'|) \leq \sum_{i=3}^{3k} \binom{n}{i}i! p^i \leq  \sum_{i=3}^{3k}  (np)^i =o(n^{0.5}). 
\end{align*}
In addition, for $v\in [n]$ deleting all the edges incident to $v$ in $G$ may decrease $|Y_k'|$  by at most  $d(v)\leq \Delta(G)$, thus by Lemma \ref{lem:AzumaHoeffding}, $|Y_k'|  \leq n^{0.55}$ with probability $1-o(n^{-2})$. Thereafter in the event   $\Delta(G)\leq \log^2n$ and $|Y_k'|\leq n^{0.55}$ we have that $|Y_k|\leq k|Y_k'|\Delta^{2k}(G) \leq n^{0.6}$. This occurs with probability at least $1-n\binom{n}{\log^2 n}p^{\log^2 n}-o(n^{-2})=1-o(n^{-2}).$ 

Finally, for $v\in [n]$ and $i\leq k$ the expected size of $N^{i}(v)$ is $c^i$. Therefore Markov's inequality implies that $\Pr(|N^i(v)|\geq 10(ck)^{3i}) \leq c^i/(10(ck)^{3i})$ and in extension that $\mathbb{E}(|Z_k|)\leq n \sum_{i=1}^k c^i/(10(ck)^{3i} \leq n/(9c^2k^3)$. Thereafter, for $v\in [n]$ deleting all the edges incident to $v$ in $G$ may decrease $|Z_k|$  by at most  $\Delta^{k}(G)$. Thus by Lemma \ref{lem:AzumaHoeffding}, $|Y_k'|  \leq n/(9c^2k^3)+ n^{0.55}\leq n/10k^2+o(n)$ with probability $1-o(n^{-2})$.

The bounds on $|X_k'|,|Y_k|$ and $|Z_k|$ and \eqref{eq:boundphi} imply,
\begin{align*}
   |\{v\in V(G):\phi_k(v)\neq \phi_k'(v) \text{ or } \phi_k'(v)\neq \phi(v)\}|\leq |X_k'|+|Y_k|+|Z_k| \leq \frac{n}{10k^2}+n^{0.6}+\frac{n}{10k^2}\leq \frac{n}{4k^2}.
\end{align*}
with probability $1-o(n^{-2}).$

\end{proof}

Lemma \ref{lem:color1} and \eqref{eq:localdiff} imply the following.
\begin{lemma}\label{lem:boundapprox}
$|L(G)-L_{c,n,k}(G)| \leq \frac{n}{4k^2}$ with probability $1-O(n^{-2})$.
\end{lemma}

\subsection{The limits of the approximations}
We now let $\mathcal{H}_{k}$ be the set of 
pairs $(H,o_H)$ where $H$ is a rooted tree, $o_H$ is a distinguished vertex of $H$ that is considered to be the root, every vertex in $V(H)$ is within distance at most $k$ from $o_H$ and there are at most $10(ck)^{3i}$ vertices at distance $1\leq i\leq k$ from $o_H$. For $(H,o_H) \in \mathcal{H}_{k}$ let $X_{(H,o_H)}(G)$ be the number of copies of $(H,o_H)$ in $G$. Also let $\phi(H,o_H)$ be equal to the value of $\phi_k(v)$ in the event $(G[N^{\leq k}(v)],v)=(H,o_H)$.
Then,
\begin{equation*}
    L_{c,n,k}(G)=n-\sum_{v \in V(G)} \phi_k(v)
    =n-\sum_{(H,o_H) \in \mathcal{H}_{k}} \phi(H,o_H) X_{(H,o_H)}(G).
\end{equation*} 
For $k\geq 1$ we let 
$$\rho_{c,k}= 1 - \sum_{(H,o_H) \in \mathcal{H}_{k}}  \frac{\phi(H,o_H) c^{|V(H)|-1}}{aut(H,o_H)}.$$
Here by $aut(H,o_H)$ we denote the number of automorphisms of $H$ that map $o_H$ to $o_H$. Then, 
\begin{align}\label{def:rho}
\mathbb{E}\bigg( \frac{L_{c,n,k}(G)}{n} \bigg) &=1-\sum_{(H,o_H) \in \mathcal{H}_{k}}\frac{ \phi(H,o_H) \mathbb{E}(X_{(H,o_H)}(G))}{n} \nonumber
\\&=1-\sum_{(H,o_H) \in \mathcal{H}_{k}} \frac{ \phi(H,o_H)  \binom{n}{|V(H)|} |V(H)|!p^{|E(H)|} (1-p)^{\binom{|V(H)|}{2}-|E(H)|}}{aut(H,o_H)\cdot n} \nonumber
\\&=1-\lim_{n \to \infty} \sum_{(H,o_H) \in \mathcal{H}_{k}} \frac{c^{|V(H)|-1}}{aut(H,o_H)} + O(n^{-0.9}) =\rho_{c,k}+ O(n^{-0.9}). 
\end{align}
\begin{lemma}\label{lem:approxrho}
With probability $1-o(n^{-2})$,
\begin{equation}\label{eq:approxrho}
    \bigg|\rho_{c,k}-\frac{L_{c,n,k}(G)}{n}\bigg|=O(n^{-0.4}).
\end{equation}
\end{lemma}
\begin{proof}
Fix $(H,o_H) \in \mathcal{H}_{k}$. By Lemma \ref{lem:AzumaHoeffding} we have that $\Pr(|\mathbb{E}(X_{H,o_H}(G))-X_{H,o_H}(G)|\geq n^{0.55})=o(n^{-2})$. As the cardinality of $\mathcal{H}_{k}$ is finite, by the union bound, we have that  $|\mathbb{E}(X_{H,o_H}(G))-X_{H,o_H}(G)|\leq n^{0.55}$ for all 
$(H,o_H) \in \mathcal{H}_{k}$  with probability $1-o(n^{-2})$. Thus,
\begin{equation*}
    \bigg|L_{c,n,k}(G)-n+\sum_{(H,o_H) \in \mathcal{H}_{k}}\phi(H,o_H)\mathbb{E}(X_{H,o_H}(G))\bigg| \leq n^{0.6}
\end{equation*}
with probability $1-o(n^{-2})$.  The above inequality combined with \eqref{def:rho} imply \eqref{eq:approxrho}.
\end{proof}

\begin{lemma}\label{lem:caushy}
For integers $k_2>k_1\geq 1$ we have,  \begin{equation}\label{eq:caushy}
    |\rho_{c,k_1}-\rho_{c,k_2}|\leq \frac{1}{2k_1^2}.
\end{equation}
\end{lemma}
\begin{proof}
Let $1\leq k_1< k_2$. Lemmas \ref{lem:boundapprox} and \ref{lem:approxrho}  imply,
\begin{align*}
    n|\rho_{c,k_1}-\rho_{c,k_2}|& \leq 
    |n\rho_{c,k_1}-L_{c,n,k_1}(G)|+|L_{c,n,k_1}(G)-L(G)|+|L_{c,n,k_2}(G)-L(G)|+|n\rho_{c,k_2}-L(G)| \nonumber
    \\& \leq O(n^{0.6})+\frac{n}{4k_1^2}+\frac{n}{4k_2^2}+O(n^{0.6}) < O(n^{0.6})+\frac{n}{2k_1^2},
\end{align*}
with probability $1-O(n^{2})$.
Thus $|\rho_{c,k_1}-\rho_{c,k_2}| \leq \frac{1}{2k_1^2}$ with positive probability for sufficiently large $n$. As $\rho_{c,k_1},\rho_{c,k_2}$ are independent of $n$ \eqref{eq:caushy} follows. 
\end{proof}
\eqref{eq:caushy} implies that the sequence $\{\rho_{c,k}\}_{k\geq 1}$ is a Caushy sequence. Therefore it has a limit as $k\to \infty$ which we denote by  $\rho_c$.

\textbf{Proof of part (a) of Theorem \ref{thm:main}}
Define $f:[0,\infty)\to[0,1]$ by $f(c)=\rho_c$ for $c\geq 20$ and $f(c)=\rho_{20}$ for $0\leq c\leq 20$.
Then for $k\geq 2$, lemmas \ref{lem:boundapprox}, \ref{lem:approxrho} and \ref{lem:caushy} imply,
\begin{align*}
    |n\rho_{c}-L_{c,n}|&\leq n|\rho_c-\rho_{c,k}|+ 
    |n\rho_{c,k}-L_{c,n,k}(G)|+|L_{c,n,k}(G)-L_{c,n}(G)|
    \\& \leq n\sum_{i\geq k}\frac{1}{2i^2}+ O(n^{0.6})+\frac{n}{4k^2}\leq \frac{n}{2(k-1)}+ O(n^{0.6})+\frac{n}{4k^2}\leq \frac{2n}{k}, 
\end{align*}
with probability $1-O(n^{-2})$. As $\sum_{i\geq 1}i^{-2}<\infty$ the Borel-Cantelli Lemma implies that 
$|\lim_{n\to \infty} 
(L_{c,n}/{n})-\rho_c|\leq 2/k$ a.s and therefore $\lim_{n\to \infty} L_{c,n}/{n}=\rho_c=f(c)$ a.s. for $c\geq 20$. 

Now let $0<\epsilon\leq 10^{-3}$. To prove that $f$ is continuous it suffices to show that $|f(c)-f(c+\epsilon)|\leq \epsilon$ for $c\geq 20$. Let $G_1\sim G(n,c/n)$,  $G_2\sim G(n,(c+\epsilon)/n)$ and $E=e_1,e_2,...,e_{2\epsilon n}$ be a sequence of $2\epsilon n$ edges where $e_i$ is chosen independently,  uniformly at random from $\binom{[n]}{2}$. Let $G_1^+=G_1\cup E$. Then $G_1,G_2,E$ can be coupled such that $L(G_2)\leq L(G_1^+)$ w.h.p., where $G_1^+$ is the simple graph obtained from $G_1\cup E$ by replacing its multiple edges with the corresponding single edges. 
We may bound $L(G_1\cup E)$ by $L(G_1)$ plus the number of vertices in components of $(G_1)^{r/b}$ that span an endpoint of an edge in $E$. Therefore, by Lemma \ref{lem:expsizecomp},
\begin{align}
\mathbb{E}(L(G_2) &\leq  \mathbb{E}(L(G_1\cup E)) \leq \mathbb{E}(L(G_1))+4\epsilon n\sum_{i=1}^\infty i \cdot \frac{0.8^in/(ci)}{n} \nonumber
\\& \leq \mathbb{E}(L(G_1))+4\epsilon n \cdot \frac{4}{c} \leq \mathbb{E}(L(G_1))+ 0.8\epsilon n. \label{eq:diffpathslength}
\end{align}

$\lim_{n\to \infty} L_{c,n}/{n}=f(c)$ a.s. implies that $\mathbb{E}(L(G_1))=nf(c)+o(n)$ and $\mathbb{E}(L(G_2))=nf(c+\epsilon)+o(n)$. Combining these equalities with \eqref{eq:diffpathslength} gives,
$$|f(c)-f(c+\epsilon)|\leq \bigg| \frac{ \mathbb{E}(L(G_1))}{n}- \frac{ \mathbb{E}(L(G_2))}{n} +o(1)\bigg| \leq 0.8\epsilon+o(1).$$
Hence $|f(c)-f(c+\epsilon)|\leq \epsilon$ as desired. 

\qed

\section{Proof of Theorem \ref{thm:ham}}\label{sec:Ham}

Fix $U \subseteq V_{blue}(G)$ a matching $M$ on  $V_{4,blue}\setminus U$ and let $H=G'[V(G')\setminus U]$. We prove Theorem \ref{thm:ham} in 3 steps. In the first one we decompose $H$ into a graph $H'\subset H$, an edge set $E_1\subset E(H)$ and a vertex set $V_1\subset V(H)$ with the following properties. $|E_1|=\Omega({n}/{\log\log n})$, $|V_1|= O({n}/{\log\log n})$ and given $V_1$,$E(H)\setminus E_1$, $|E_1|$ the set $E_1$ is uniformly distributed over all the sets of edges of size $|E_1|$ that are spanned by  $V(H)\setminus V_1$ and are disjoint from $E(H)\setminus E_1$. Then, by applying the Tutte-Berge formula twice, we find a set of pairwise disjoint vertex paths in $H\cup M$ of size at most $4{n}/{\log^{0.5}n}$ that cover both $V(H)$ and $M$. Finally, using P\'osa rotations we merge these paths into a Hamilton cycle that covers $M$. 

\subsection{Decomposing \texorpdfstring{$H$}{Lg}}

To decompose $H$ we first assign to every edge $e$ of $G$ a $Bernoulli(p')$ random variable $Y_e$ with $p'= 1/c\log\log n$. Then we let $H_1$ be the subgraph of $H$ with edge set $E(H_1)=\{ e\in E(H): Y_e=0\}$ and we reveal $H_1$. Thereafter, given $V_{red}(G)$ we identify $V_{black}(G)$ and let $V_1$ be the set of vertices of $V(H)$ with less than $4$ neighbors in $V_{black}(G)$. Finally we reveal all the edges of $H$ incident to $V_1$, define $H'$ by $V(H')=V(H)$ and $E(H')=E(H_1) \cup \{uv\in E(H): \{u,v\} \cap V_1 \neq \emptyset\}$ and let $E_1=E(H)\setminus E(H')$.  

Given $H', V_1$ and $e_1=|E_1|$ let $\cS(H',V_1,e_1)$
be the set that consists of all the sets of edges $T$ that are spanned by $V(H)\setminus V_1$, do not intersect $E(H')$ and have size $e_1$. Observe that $\Pr(T=E_1|H',V_1, e_1)=0$ for $T\notin \cS(H',V_1,e_1)$. On the other hand for $T\in \cS(H',V_1,e_1)$ we have that $\Pr(T=E_1|H',V_1, e_1)$ is independent $T$. Hence the distribution of $E_1$ is uniform over the elements of $\cS(H',V_1,e_1)$. The sizes of $V_1$ and $E_1$ are given by the following lemma. Its proof is located in Appendix \ref{sec:app:sizev1e1}.

\begin{lemma}\label{lem:sizev1e1}
Let $\cE_{sample}$ be the event that $|V_1|\leq {10n}/{\log\log n}$ and ${n}/{1000\log\log n}\leq |E_1|$. Then, $\Pr(\cE_{sample})=1-o(n^{-2})$.
\end{lemma}

\subsection{Finding a large 2-matching}
For integers $k,\ell, r$, we say that a graph $F$ has the property $\cP(k,\ell,r)$, equivalently $F\in \cP(k,\ell,r)$, if the following hold. $F$ spans at most $k$ vertex disjoint cycles of length at most $\ell$ and  there does not exist a partition of $V(F)$ into 3 pairwise disjoint sets $U_1,U_2,U_3$ such that $|U_1|> r$, $|U_2|\leq |U_1|$ and every vertex in $U_1$ has at most 1 neighbor in $U_1\cup U_3$.

\begin{lemma}\label{lem:tutte}
Let $F$ be a graph and $\ell,k,r$ be such that $F\in \cP(k,\ell,r)$. Then, for every matching $M$ on $V(F)$ the graph $F\setminus M$ spans a matching $M'$ of size at least $0.5|V(F)|-0.5(r+k+({|V(F)|}/{\ell}))$.
\end{lemma}

\begin{proof}
For a graph $G$ and $U\subset V(G)$ let $odd_G(U)$ be the number of odd components of $G[V(G)\setminus U]$. In addition denote by $\alpha'(G)$ the matching number of $G$. The Tutte-Berge formula states
\begin{equation}\label{eq:TB}
 \alpha'(G)=0.5 \min_{U\subseteq V(G)}(|U|-odd_G(U) +|V(G)|).   
\end{equation}
Let $S\subseteq V(F)$ be a set that minimizes $|S|-odd_{F\setminus M}(S)$ of maximum size, $A$ be the set of isolated vertices in $F[V(F)\setminus S]\setminus M$ and $B=V(F)\setminus (S\cup A)$. Observe that $B$ does not span a tree in $F\setminus M$. Indeed, assume otherwise, that is that $B$ spans a tree $T$. Let $l$ be a leaf of $T$ and $p$ the parent of $l$ (in the case that $T$ consists of a single edge $e$ we let $p,l$ be the endpoints of $e$). If $|V(T)|$ is even then $T\setminus \{p\}$ spans at least one odd component, namely the one consisting of the vertex $l$. Else if $|V(T)|$ is odd then  $T\setminus \{p\}$ spans at least one odd component in addition to $\{l\}$, hence at least 2. Therefore  with $S'=S\cup \{p\}
$,
$$|S'|-odd_{F\setminus M}(S') \leq (|S|+1) - (odd_{F\setminus M}(S)+1)= |S|-odd_{F\setminus M}(S)$$  contradicting the maximality of $S$. Therefore every component spanned by $B$ contains a cycle. $F\in \cP(k,\ell,r)$ implies that there exist at most $k$ cycles of length at most $\ell$ in $F$, hence in $F\setminus M$ and therefore $B$ spans at most $k+{|B|}/({\ell+1})\leq k+{|V(F)|}/{\ell}$ many components. In this case \eqref{eq:TB}  and the choice of $S$ imply that 
\begin{equation}\label{eq:matchA}
    \alpha'(F\setminus M)= 0.5|V(F)|+0.5|S|-0.5odd_{F\setminus M}(S)
    \geq  0.5|V(F)|+0.5|S| -0.5(|A|+k+(|V(F)|/\ell)).
\end{equation}
If $|A|\leq r$ then \eqref{eq:matchA} gives that $\alpha'(F\setminus M)\geq 0.5|V(F)|-0.5( r+k+(|V(F)|/\ell))$. On the other hand if $|A|>r$ then every vertex in $A$ has no neighbor in $A\cup B$ in $F\setminus M$ and therefore it has at most one neighbor in $A\cup B$ in $F$. Hence, as $F$ belongs to $ \cP(k,\ell,r)$ (with $(A,S,B)=(U_1,U_2,U_3)$) we have that $|S| \geq |A|$. In this case \eqref{eq:matchA} gives that $\alpha'(F\setminus M)\geq 0.5|V(F)|-0.5(k+(|V(F)|/\ell))$.
\end{proof}

We prove the following lemma in Appendix \ref{sec:app:find2match}.

\begin{lemma}\label{lem:find2match}
Let $U' \subseteq V_{blue}(G)\setminus U$. Then both
$H'$ and $H'[V(H')\setminus U']$ belong to $\cP({n}/{\log^{0.5}n}, \log^{0.5} n, 0 )$ with probability $1-o(n^{-2})$. In addition $H'\cup M$ does not span a set of $n/\log^{0.5}n$ vertex disjoint cycles of length at most $\log^{0.5}n$ with probability $1-o(n^{-2})$.
\end{lemma}

By combining lemmas \ref{lem:tutte}  and \ref{lem:find2match} we prove the following one.
\begin{lemma}\label{lem:paths}
There exists a set of vertex disjoint paths $\cP$ in $H'\cup M$ of size at most $4n/\log^{0.5} n$ that cover both $V(H')$ and $M$ with probability $1-o(n^{-2})$.  
\end{lemma}

\begin{proof}
Let $M_1$ be a maximum matching in $H'\setminus M$, $M_1^+=M\cup M_1$,
$V_M$ be the set of vertices that are incident to 2 edges in $M_1^+$  and $M_2$ be a maximum matching in $H'[V(H')\setminus V_M]\setminus M_1^+$. To construct the set $\cP$, let $\cC$ be the set of components induced by $M_1^+\cup M_2$. Remove from every cycle in $\cC$ an edge that does not belong to the matching $M$ and let $\cP$ be the set of the resulting $|\cC|$ paths.

Let $\cE$ be the event that $H', H'[V(H')\setminus V_M]\in \cP(n/\log^{0.5}n,\log^{0.5}n,0)$  and $H'\cup M$ does not span a set of $n/\log^{0.5}n$ vertex disjoint cycles of length at most $\log^{0.5}n$.  In the event $\cE$, by Lemma \ref{lem:tutte}, $|M_1|\geq 0.5|V(H')|- n/(\log^{0.5} n)$ and $|M_2|\geq 0.5(|V(H')|-|V_M|)- n/(\log^{0.5} n)$.  Therefore the components in $\cC$ span  at least $|V(H')|-2n/\log^{0.5}n$ edges in total. In addition, as every cycle in $\cC$ belongs to $H'\cup M$, $\cC$ spans at most $n/\log^{0.5}n$ cycles of length less than $\log^{0.5}n$ and $2n/\log^{0.5}n$ cycles in total. This implies that $\cP$ is a set of vertex disjoint paths that covers $|V(H')|$ and spans at least $|V(H')|-4n/\log^{0.5}n$ edges. Thus
$$|V(H')|-4n/\log^{0.5}n\leq \sum_{P\in \cP}|E(P)|=\sum_{P\in \cP}|V(P)|-1 =|V(H')|-|\cP|.$$
Hence  $|\cP|\leq 4n/\log^3 n$ with probability $\Pr(\cE)=1-o(n^{-2})$.  
\end{proof}

\subsection{Merging the paths into a Hamilton cycle}
Let $\cP=\{P_1',P_2',...,P_{\ell}'\}$ be a minimum size set of vertex disjoint paths that cover both $M$ and $V(H')$. For $i\in [\ell]$ let $v_{i,1},v_{i,2}$ be the two endpoints of $P_i'$ (in the case that $P_i'$ is a path of length $0$, equivalently it consists of a single vertex $v_i$,  then $v_{1,i}=v_i=v_{2,i}$). Then $P_1= P_1',v_{1,2}v_{2,1},P_2',v_{2,2}v_{3,1},....,$ $P_\ell'$ is a Hamilton path of $H'\cup R$ where $R=\{v_{i,2}v_{i+1,1}:i\in [\ell-1]\}$. We transform $P_1$ into a Hamilton cycle of $H=H'\cup E_1$ in $\ell$ iterations of an extension-rotation procedure. Given a Hamilton path $P=v_1,e_1,v_2,e_2,...,v_i,e_i,v_{i+1},...,e_{n'-1},v_{n'}$ we say that the path $P'=v_1,e_1,....,v_i,v_iv_{n'},v_{n'},e_{n'-1},v_{n'-1},....,e_{i+1},v_{i+1}$ is obtained by a P\'osa rotation with $v_1$ being the fixed endpoint. We call $e_i$ the deleted edge, $v_1v_{n'}$ the inserted edge, $v_{i}$ the pivot vertex and $v_{i+1}$ the new endpoint. We say that the P\'osa rotation that transforms $P$ to $P'$ is admissible w.r.t. to the pair of edge sets ($W$, $W'$) if the inserted edge belongs to $W$ and the deleted edge does not belong to $W'$.   

Let $\ell=|\cP|$. For $i\in [\ell]$ we let $E_i'$ be the set of edges in $E_1$ that have been revealed during the first $i-1$ iterations, thus $E_1'=\emptyset$. We start the $i^{th}$ iteration with a Hamilton path $P_i$ in $H'\cup E_i'\cup R$ that spans $\ell-i$ edges of $R$. We then proceed by performing all sequences of P\'osa rotations that fix the vertex $v$ and are admissible w.r.t. ($E(H')$,$M$) (each such sequence starts with the path $P_i$). Let $End_i$ be the set of distinct new endpoints obtained and for $w\in End_i$ let $P_{w,i}$ be a path from $v$ to $w$  obtained by the P\'osa rotations. Thereafter, for $w\in End_i$ we perform all sequences of P\'osa rotations that fix the vertex $w$ and are admissible w.r.t. ($E(H')$,$M$) (each such sequence starts with the path $P_{w,i}$) and we let $End_{w,i}$ be the set of distinct new endpoints obtained.

For $w\in End_i,z\in End_{w,i}$ we let $P_{\{w,z\},i}$ be a path from $w$ to $z$ obtained by the above procedure. If there exists a path $P_{\{w,z\},i}$ that contains fewer edges in $R$ than $P_i$ then we let $P_{i+1}$ be such a path that spans $\ell-i-1$ edges in $R$, set $E_{i+1}'=E_i'$ and proceed to the next iteration. Else, we reveal the edges in $E_1\setminus E_i'$ one by one until we identify an edge $w,z$ with $w\in End_i$, $z\in End_{w,i}$. Once such an edge is identified, we let $H_i$ be the Hamilton cycle with edge set $E(P_{\{w,z\},i})\cup \{\{w,z\}\}$. If $i=\ell$ then we output $H_\ell$. Else, $H_i$ spans $\ell-i-1$ edges in $R$, we remove such an edge from $H_i$ and let $P_{i+1}$ be the resultant Hamilton path. If at any point we have revealed all the edges in $E_1$ and have not constructed $H_\ell$ yet, then we terminate the algorithm.

For $e\in E_1$ set $X_e=1$ if $e$ is not revealed by the above algorithm or when $e$ is revealed it is used to construct some Hamilton cycle $H_i,i\leq \ell$.  Set $X_e=0$ otherwise. All  P\'osa rotations performed by the above algorithm are admissible w.r.t. $(E(H'),M)$, thus they never delete an edge from $M$ or add an edge from $R$ to a path while they are performed. Here we are using that $\cP$ is of minimum size, hence $R\cap E(H')=\emptyset$. So in the event $\sum_{e\in E_1} X_e\geq \ell$, $H_\ell$ is a Hamilton cycle of $H'\cup E_1\cup R= H\cup R$ that spans at most $|R|-\ell=0$ edges in $R$ and all of the edges of $M$.

Let $\cE_{exp}$ be the event the following hold: (i) every set $W\subset [n]$ of size at most $12$ spans at most $|W|+2$ edges in $H'$, (ii) for every $S\subset[n]$ and $T\subseteq[n]\setminus S$ with $5\leq |S|\leq n/c^5$ and $|T|\leq 2|S|$ we have that the set $S\cup T$ spans fewer than $1.5|S|+|T|$ edges in $H'$ and (iii) for every set $S\subset [n]$ satisfying $n/c^9\leq|S|\leq 10^{-30}n$ we have that $|N_{H'}(S)\setminus V_{blue}(G)|\leq 2|S|$.  In the analysis of the above algorithm we make use of the following lemma. 
\begin{lemma}\label{lem:expansion} 
$\Pr(\cE_{exp})= 1-O(n^{-2})$.
\end{lemma}
\begin{proof} As $H'\subseteq G$ we have,
$$\Pr(\neg (i))\leq \sum_{s=4}^{12}  \binom{n}{s} \binom{0.5s^2}{s+3}p^{s+3}=O(n^{-2}).$$
In addition,
\begin{align*}
\Pr(\neg (ii))&\leq
\sum_{s=5}^{\frac{n}{c^9}} \sum_{t=0}^{2s}  \binom{n}{s+t} \binom{0.5(s+t)^2}{1.5s+t}p^{1.5s+t}
\leq \sum_{s=5}^{\frac{n}{c^9}} \sum_{t=0}^{2s}  \bfrac{en}{s+t}^{s+t} \bfrac{0.5e(s+t)^2p}{1.5s+t}^{1.5s+t}  
\\&\leq O(n^{-2}) + \sum_{s=\log^2 n}^{\frac{n}{c^9}} \sum_{t=0}^{2s}  \bfrac{0.5e^2np}{1.16}^{s+t} \bfrac{0.5e(s+t)p}{1.16}^{0.5s} 
\\& \leq O(n^{-2}) + \sum_{s=\log^2 n}^{\frac{n}{c^9}} 2s  \bfrac{e^2c}{2.32}^{3s} \bfrac{3e}{2.32c^9}^{0.5s}
 \leq O(n^{-2}) + \sum_{s=\log^2 n}^{\frac{n}{c^9}} s \bfrac{3e^{13}}{2.32^7c^3}^{0.5s} = O(n^{-2}). 
\end{align*}
For $c\leq 1000$, as $n/c^9> 10^{-30}n$, we have that $\Pr( \cE_{exp})=1-O(n^{-2})$. Thereafter for $c>1000$ and $s\geq n/c^9$ Lemma \ref{lem:expsizecomp}  implies that $|V_{blue}(G)\cup V_{red}(G)|\leq c^3e^{-c}n\leq n/c^9 \leq 0.1s$ with probability $1-O(n^{-2})$. In addition by construction each edge $e$ spanned by $V_{black}(G)$ does not belong to $H'$ only if (i) $e\notin E(G)$ or  (ii) $e\in E(G)$ and $Y_e=1$, hence with probability at most $1-p+pp'$ independently. Thus,
\begin{align*}
\Pr(\neg \cE_{exp})&\leq O(n^{-2})+\sum_{s=\frac{n}{c^9}}^{10^{-30}n}  \binom{n}{s}\binom{n}{2.1s}(1-p+pp')^{s(n-3.1s)}
\\& \leq O(n^{-2})+\sum_{s=\frac{n}{c^9}}^{10^{-30}n}  \bfrac{en}{s}^s\bfrac{en}{2.1s}^{2.1s}e^{-(1+o(1))ps\cdot 0.999n}
\\&\leq  O(n^{-2})+ \sum_{s=\frac{n}{c^9}}^{10^{-30}n}  \bigg[5\bfrac{n}{s}^{3.1s}e^{-0.99c}\bigg]^s\leq O(n^{-2})+ \sum_{s=\frac{n}{c^9}}^{10^{-30}n} \bigg[ c^{30}e^{-0.99c}\bigg]^s=O(n^{-2}). 
\end{align*}
\end{proof}

Theorem \ref{thm:ham} follows from Lemma \ref{lem:Haml}.
\begin{lemma}\label{lem:Haml} 
 $\sum_{e\in E_1} X_e\geq \ell$ with probability $1-O(n^{-2})$.
\end{lemma}
\begin{proof}
Let $\cE$ be the event that the events $\cE_{sample}$ and $\cE_{exp}$ occur, and there exists a set of at most 
$4n/\log n$ vertex disjoint paths in $H'\cup M$ that cover both $M$ and $V(H')$. Lemmas \ref{lem:sizev1e1}, \ref{lem:paths} and  \ref{lem:expansion} imply that $\Pr(\cE)=1-O(n^{-2})$. 

Let $P$ be any Hamilton $u$-$v$ path in $H'\cup E_1\cup  R$. Recall that $u$ has at least 4 neighbors in $V_{black}(G)$ in $H'$. At most 1 of those neighbors precedes $u$ on $P$. Also as these neighbors belong to $V_{black}(G)$ they are not incident to $M$ which is spanned by $V_{blue}(G)$. Therefore there are at least 3 admissible P\'osa rotations w.r.t $(E(H'),M)$ that can be performed on $P$ and fix $v$. If none of the corresponding deleted edges belongs to $R$, then the corresponding, at least $3$, pivot vertices are adjacent to $u$ and the corresponding new endpoint in $H'$. Call this observation $(*)$.

Let $i\in [\ell]$ and $P_i,v,End_i,\{End_{w,i}:w\in End_i\},\{P_{w,i}:w\in End_i\}$, $\{P_{\{w,z\},i}:w\in End_i,z\in End_{w,i}\}$ be as described earlier. Assume that at iteration $i$ we do not perform a P\'osa rotation where the deleted edge belongs to $R$. Let $Pivot_i$ be the set of pivot vertices that we meet while constructing the set $End_i$ (starting from $P_i$). $(*)$ implies that, in $H'$, every vertex in $End_i$ is adjacent to at least 3 vertices in $Pivot_i$ and every vertex in $Pivot_i$ is adjacent to at least 2 vertices in $End_i$. It follows that the set $End_i\cup Pivot_i$ spans at least $1.5|End_i|+|Pivot_i\setminus End_i|$ many edges in $H'$. If $4\leq |End_i|\leq n/c^9$, by considering a first time a vertex in $Pivot_i$ is used as a pivot vertex, every vertex in $Pivot_i$ has a neighbor on $P_i$ that belongs to $End_i$, hence $|Pivot_i\setminus End_i|\leq 2|End_i|$. In the special case that $|End_i|=4$, let $v,u$ be the endpoints of $P_i$, $End_i=\{u,u_1,u_2,u_3\}$ where $u_3$ is the vertex further from $u$ on $P_i$ and $w_j$ be the vertex preceding $u_j$ on $P_i$ for $j=1,2,3$. $(*)$ states that there are at least 3 admissible P\'osa rotations w.r.t $(E(H'),M)$ that can be performed on $P_{i}$ and fix $v$. As $End_i=\{u,u_1,u_2,u_3\}$ and no two of these P\'osa rotations result in a pair of paths with the same endpoints we have that $vw_i\in E(H')$ for $i=1,2,3$. Let $P_{i,j}$ be the path from $v$ to $u_j$ that can be obtained by a single P\'osa rotation from  $P_i$. Observe that on both $P_{i,1}$, $P_{i,2}$ the vertex $w_3$ precedes $u_3$ (as we traverse them starting from $v$). Once again, as $End_i=\{u,u_1,u_2,u_3\}$, $(*)$ implies that $u_1w_3,u_2w_3\in E(G)$. Thus $w_3\in Pivot_i\setminus End_i$ has $4$ neighbors in $End_i$. It follows that  $Pivot_i\cup End_i$ has size $s\in [4,12]$ and spans at least $1.5|End_i|+|Pivot_i\setminus End_i|+1=s+3$ many edges in  $H'$. 

Partition $N_{H'}(End_i)$ to $N_1\uplus N_2$ where $N_1$ is the set of vertices in $N_{H'}(End_i)$ that have a neighbor on $P_i$ who belongs to $End_i$. Then $|N_1|\leq 2|End_i|$. Let $u\in N_2=N_{H'}(End_i)\setminus N_1$, say $u \in N_{H'}(w)$ with $w\in End_i$.  As none of the neighbors of $u$ on $P_i$ belong to $End_i$, the P\'osa rotation that inserts to $P_{w,i}$ the edge $uw$ and deletes an edge incident to $u$ is not admissible w.r.t $(E(H'),M)$. Thus $u$ is incident to an edge in $M$. As every edge in $M$ is spanned by $V_{blue}(G)$ we have that $u\in V_{blue}(G)$. It follows that
$$|N_{H'}(End_i)\setminus V_{blue}(G)|= |N_1|\leq 2|End_i|.$$
In the event $\cE$, the event $\cE_{exp}$ occurs. By taking $S=End_i$, $T=Pivot_i\setminus End_i$ and $W=End_i\cup Pivot_i$ at the definition of $\cE_{exp}$ the above imply that in the event $\cE$ we have that $|End_i|\geq 10^{-30}n$ and (similarly) $|End_{w,i}|\geq 10^{-30}n$ for $w\in End_i$. Hence at iteration $i$ there exists a set of at least $(0.5+o(1))10^{-60}n^2$ pairs $\{w,z\}\subset V(H)\setminus V_1$ such that during iteration $i$ the Hamilton path $P_{\{w,z\},i}$ is generated. Here we are using that in the event $\cE$ the set $V_1$ spans $o(n^2)$ pairs of  vertices. Thus for every edge $e\in E_1$ that is revealed at iteration $i$ we have that $X_e=1$ with probability at least $(1+o(1))10^{-60}$ independently of the identity of the edges in $E_1$ that are revealed beforehand. It follows that, 
\begin{align*}
\Pr\bigg(\sum_{e\in E_1} X_e < \ell\bigg)\leq \Pr\bigg(Bin\bigg(\frac{n}{1000\log\log n},(1+o(1))10^{-60}\bigg)\leq \frac{4n}{\log^{0.5}n}\bigg|\cE \bigg)+\Pr(\neg \cE) =O(n^{-2}).
\end{align*} 
\end{proof}

\section{Concluding Remarks}
We have shown how one can identify a longest cycle in $G\sim G(n,c/n)$ and proved that $\lim_{n\to \infty} L(G)/n$ converges to a constant $f(c)$ a.s. for $c\geq 20$. In addition we determined the probability that $G$ is weakly pancyclic. Our proofs rely on structural properties of the strong $4$-core of the binomial random graphs that hold with high probability. This motivates the further study of the strong $k$-core of $G(n,p)$ and in particular determining $np_k$, where $p_k$ is the threshold of appearance of a non-empty strong $k$-core for $k\geq 3$ in $G(n,p)$. We believe that one can extend the range for which Theorem \ref{thm:main} holds to $c> np_3$ via the study of the strong $3$-core.

\subsection*{Acknowledgment}  
We would like to thank the reviewers for their helpful comments and remarks.

\bibliographystyle{plain}
\bibliography{bib}

\begin{appendix}
\section{Proof of Lemma \ref{lem:sizev1e1}}\label{sec:app:sizev1e1}

\begin{proof}
Let $F'$ be the subgraph of $G$ induced by $\{e\in G: Y_e=1\}$. Then
$F'\sim G(n,1/n\log\log n)$. Therefore,
\begin{align*}
    \Pr\bigg(|V_1|> \frac{10n}{\log\log n}\bigg)
    &\leq \Pr\bigg(|E(F')|> \frac{5n}{\log\log n}\bigg)
    \\&\leq \Pr\bigg(Bin\bigg(n,\frac{1}{n\log\log n} \bigg)> \frac{5n}{\log\log n} \bigg)+o(n^{-2})=o(n^{-2}). 
\end{align*}
Let $\cE$ be the event that $H$ spans at least $0.25cn$ edges.  Lemma \ref{lem:expsizecomp} states that $|V_{red}(G)\cup V_{blue}(G)|\leq c^3e^{-c}n\leq 0.01n$ with probability $1-o(n^{-2})$. As every edge spanned by $V_{black}(G)$ belongs to $H$,
\begin{align*}
    \Pr(\neg \cE)&\leq  \binom{n}{0.01n}\Pr\bigg(Bin\bigg(\binom{0.99n}{2},\frac{c}{n}\bigg)<0.25cn\bigg)+o(n^{-2})\leq 2^n e^{-\frac{0.49^2\cdot 0.495cn}{2}}+o(n^{-2})=o(n^{-2}).
\end{align*}
Furthermore, let $E_5'=E_5\cap E(H)$ and $E_2'$ be the set of edges are incident to vertices of degree at least 2 in $F'$. Observe that every edge of  $E_5'\setminus E_2'$ belongs to $E_1$. Thus $|E_1| \geq {n}/{1000\log\log n}$ if $|E_5'| \geq {n}/{200\log\log n}$ and  $|E_2'|\leq {n}/{400\log\log n}$. In the event $\cE$ we have that $|E_5|\geq |E(H)|-4n\geq  |E(H)|-0.2cn \geq  0.05cn$. It follows that,
\begin{align*}
   \Pr\bigg(|E_5'| < \frac{n}{200\log\log n} \bigg) 
   \leq \Pr\bigg(Bin\bigg(0.05cn,\frac{1}{c\log\log n}\bigg) \leq \frac{n}{200\log\log n}\bigg |\cE \bigg)+o(n^{-2}) =o(n^{-2}).
\end{align*}
For upper bounding $E_2$, let $X_i$ be the number of vertices of degree $i$ in $F'$. By Lemma \ref{lem:AzumaHoeffding} we have that $X_i\leq n \binom{n}{i}(pp')^i +n^{0.6}$ for $i\geq 0$ with probability $1-o(n^{-2})$. In addition,
$$\Pr(\Delta(F')\geq \log^2n)\leq \Pr(\Delta(G)\geq \log^2n) \leq n \Pr(Bin(n,2\log n/n)\geq \log^ 2n)=o(n^{-2}).$$
Hence, $|E_2'| \leq \sum_{i= 2}^{\log^2 n} i |X_i|\leq 
\sum_{i= 2}^{\log^2 n} n 2^{i+1}(\log\log n)^{-i}+ n^{0.6} \leq  {n}/{400\log\log n}$ with probability $1-o(n^{-2})$. It follows that $|E_1|\geq |E_5'|-|E_2'|\geq n/(1000\log\log n)$ with probability $1-o(n^{-2})$.

\end{proof}

\section{Proof of Lemma \ref{lem:find2match}}\label{sec:app:find2match}

\begin{proof}
Let $Z$ and $Z^+$ respectively be the maximal number of vertex disjoint cycles of length at most $\log^{0.5}n$ in $H'\cup M$ and at most $\log^{0.6}n$ in $G$ respectively. Say a cycle in $H'\cup M$ is heavy if it spans an edge in $M$ that corresponds to a path in a component of $G^{r/b}$ with more than $\log^{0.1}n$ vertices. Let $Z_H$ be the number of heavy cycles in $H'\cup M$ and $X_{\geq \log^{0.1}n}$ be the number of components of $G^{r/b}$ of size at least $\log^{0.1}n$.  Then $Z\leq Z^+ +Z_H\leq Z^+ + X_{\geq \log^{0.1}n}$. 
Thereafter using that $Z^+$ is bounded by the number of cycles of length at most $\log^{0.6}n$ and Lemma \ref{lem:expsizecomp} we have,
$$\mathbb{E}(Z) \leq \sum_{i=3}^{\log^{0.6}n} \binom{n}{i} \frac{(i-1)!}{2}p^i + n(0.8)^{\log^{0.1}n} \leq \sum_{i=3}^{\log^{0.6}n} (np)^i + n(0.8)^{\log^{0.1}n}\leq \frac{n}{\log n}.$$
By adding/deleting a single edge $Z$ may increase/decrease by at most 1. Thus, by Lemma \ref{lem:AzumaHoeffding}, $\Pr(Z\geq n/\log^{0.5}n) =o(n^{-2})$. As both graphs $H'$ and $H'[V(H')\setminus U']$ are subgraphs of $H'\cup M$ we have that none of these 3 graphs contains a union of $n/\log^{0.5}n$ vertex disjoint cycles of length at most $\log^{0.5}n$.

Given the above, to prove that Lemma \ref{lem:find2match} it suffices to prove that with probability $1-o(n^{-2})$ the following hold:
\begin{itemize}
    \item[(P1)] There does not exists a pair of sets $S,T$ of size $|S|=|T|\leq {n}/{1.25c^3}$ and every vertex in $S$ has at least 3 neighbors in $T$ in $G$.
    \item[(P2)] There do exist sets $S,R\subset V(G)$ of size $|R|\leq |S|\in [{n}/{1.25c^3},0.3n]$ such every vertex in $S$ has at most $1$ neighbor not in $R\cup V_{blue}(G)\cup V_{red}(G)$ in $G\setminus E_1$.
    \item[(P3)] There does not exists a set $U\subset [n]$ of size $0.3n$ such that  $G[U]\setminus E_1$ is a matching.
\end{itemize}
Indeed let $F\in \{H', H'[V(H')\setminus U]\}$ and assume that $V(F)$ has a partition into  pairwise disjoint sets $U_1,U_2,U_3$ such that $|U_2|\leq |U_1|$ and in $F$ every vertex in $U_1$ has at most 1 neighbor in $U_1\cup U_3$, hence in $U_1$. As every vertex in $U_1$ has at most one neighbor in $U_1\cup U_3$ and at least 4 neighbors in $V_{black}(G)\subseteq V(F)$ it must have at least 3 neighbors in $U_2$. Thus, as $F=G[V(F)]\setminus E_1$, if (P1) and (P3) hold then  $n/(1.25c^3)\leq |U_1|\leq 0.3n$. Thereafter as $V_{black}(G)\subseteq V(F)$, every vertex in $U_1$ has at most $1$ neighbor in $(U_1\cup U_3)\cap V_{black}(G)$ in $F$ hence at most $1$ neighbor in $G\setminus E_1$ that does not lie in $U_2\cup V_{red}(G)\cup V_{blue}(G)$. Thus if (P1) and (P3) hold then (P2) does not hold. 

We now bound $\Pr(\text{P1}),\Pr(\text{P2})$ and $\Pr(\text{P3})$.

\begin{align*}
\Pr(\text{P1})\leq \sum_{s=3}^{\frac{n}{1.25c^3}} \binom{n}{2s}\binom{2s}{s} \bigg(\binom{s}{3}p^3\bigg)^s \leq \sum_{s=3}^{  \frac{n}{1.25c^3}} \bfrac{en}{2s}^{2s} 2^{2s} \bfrac{(sp)^3}{6}^s  \leq  \sum_{s=3}^{  \frac{n}{1.25c^3}} \bfrac{e^2 c^3 s}{6n}^{s} =o(n^{-2}).
\end{align*}

Let $\cE$ be the event that $|V_{blue}(G)\cup V_{red}(G)|\leq c^3e^{-c}n$. By Lemma \ref{lem:expsizecomp},  $\Pr(\cE)= 1-o(n^{-2})$. In the event $\cE$, if (P2) holds then $V(G)$ spans a pair of disjoint sets $S,R'\subset [n]$  such that (i) $|S|\in [n/1.25c^3,0.3n]$, (ii) $|R'|\leq |S|+c^3e^{-c}n$, (iii) every vertex in $R'$  has a neighbor in $S$ in $G$, and (iv) every vertex in $S$ has at most $1$ neighbor in $V(G)\setminus R'$ in $G\setminus E_1$. Here we may substitute conditions (ii) and (iii) with the weaker condition (v) $|R'|= |S|+c^3e^{-cn}$, this is done for bounding $p_2$ below.
Thereafter, for $c\geq 20$ we have that if $n/1.25c^3\leq |S|\leq 0.01n$ then $c^3e^{-c}n\leq 0.17|S|$, else if $|S|\geq 0.01n$ then $c^3e^{-c}n\leq 0.01|S|$. Recall that each edge $e$ does not belong to $G\setminus E_1$ only if (i) $e\notin E(G)$ or  (ii) $e\in E(G)$ and $Y_e=1$, hence with probability at most $1-p+pp'$ independently. Thus for $c\geq 20$,  $\Pr(\text{P2})\leq p_1+p_2$ where,
\begin{align*}
p_1 &\leq \sum_{s=\frac{n}{1.25c^3}}^{0.01n} \sum_{r=0}^{1.17s}
\binom{n}{s}\binom{n}{r}(sp)^{r} ((n-r)p+1)^s  (1-p+pp')^{s(n-(s+r))}
\\& \leq \sum_{s=\frac{n}{1.25c^3}}^{0.01n}  \sum_{r=0}^{1.17s}\bfrac{en}{s}^s \bfrac{esnp}{r}^{r}  (np+1)^s e^{-0.975spn} 
\\&\leq n\sum_{s=\frac{n}{1.25c^3}}^{0.01n} \bigg[\bfrac{en}{s} \bfrac{esnp}{1.17s}^{1.17} (np+1) e^{-0.975pn}  \bigg]^s
\\&\leq n^2\bigg[(1.25ec^3)(ec)^{1.17}c e^{-0.975c}\bigg]^s 
=o(n^{-2}) 
\end{align*}
and
\begin{align*}
p_2 &\leq \sum_{s=0.01n}^{0.3n} \binom{n}{2.01s}\binom{2.01s}{s}((n-2.01s)p+1)^s  (1-p+pp')^{s(n-2.01s)}
\\&\leq \sum_{s=0.01n}^{0.3n} \bigg[\bfrac{en}{2.01s}^{2.01} 2^{2.01}  (c(1-2.01s/n)+1) e^{-(1+o(1))c(1-2.01s/n)}\bigg]^s
=o(n^{-2}).
\end{align*}
Finally,
\begin{align*}
\Pr(\text{P3})&\leq \binom{n}{0.3n} \sum_{s=0}^{0.15n} \binom{0.3n}{2s} \frac{(2s)!}{s!2^s}p^s (1-p+pp'))^{\binom{0.3n}{2}-s} 
\\& \leq 2^{-(0.3\log_2 0.3 +0.7\log_2 0.7+o(1))n} \sum_{s=0}^{0.15n} 2^{0.3n} \bfrac{2sp}{e}^s e^{-p\binom{0.3n}{2}} \leq 2^{1.2n} \sum_{s=0}^{0.15n} \bfrac{2sp}{e}^s e^{-p\binom{0.3n}{2}}
\\& \leq n\bigg[2^{4} \bfrac{0.15c}{e}^{0.5} e^{-0.15c} \bigg]^{0.3n} \leq n\bigg[2^{4} \bfrac{3}{e}^{0.5} e^{-3} \bigg]^{0.3n}  =o(n^{-2}).
\end{align*}

\end{proof}

\end{appendix}

\end{document}